\documentclass[leqno,twoside,11pt]{amsart}
\usepackage{latexsym, esint, url, amssymb, yfonts, amsmath}
\usepackage{graphicx,color}
\usepackage{tikz-cd}
\usepackage{caption}
\usepackage{subcaption}
\usepackage{array}

\theoremstyle{plain}

\newtheorem{theorem}{Theorem}[section]
\newtheorem{lemma}[theorem]{Lemma}

\newtheorem{corollary}[theorem]{Corollary}
\newtheorem{definition}[theorem]{Definition}

\theoremstyle{definition}

\numberwithin{equation}{section}
\numberwithin{figure}{section}

\def\R{{\mathbb R}}

\definecolor{Green}{rgb}{0, 0.65,0}

\begin{document}
\title[A Stability Result for the $\infty$-Laplace Equation]{A Stability Result for the $\infty$-Laplace Equation}
\author{Marta Lewicka}
\address{University of Pittsburgh, Department of Mathematics,
301 Thackeray Hall, Pittsburgh, PA 15260, USA }
\author{Nikolai Ubostad}
\address{Department of Mathematical Sciences 
Norwegian University of Science and Technology 
N-7491 Trondheim, Norway}
\email{lewicka@pitt.edu, nikolai.ubostad@ntnu.no}
\date{\today}

\begin{abstract}
We investigate a degenerate elliptic PDE related to the
$\infty$-Laplace equation $\Delta_{\infty}u=0$.  A stability result is derived via Jensen's Auxiliary  equations. The $\Gamma$-convergence of the corresponding functionals is proven.
\end{abstract}

\maketitle

\section{Introduction}

The $\infty$-Laplace equation:
\begin{equation} \label{eq:introinfty}
\Delta_{\infty}u =\sum_{i, j =1}^{n}u_{x_i}u_{x_j}u_{x_ix_j}=0 
\end{equation}
was introduced by Aronsson in \cite{aronsson}. Its solutions are
called the $\infty$-harmonic functions and they turn out to be the
absolutely minimizing Lipschitz extensions of the given boundary
values. In particular, equation (\ref{eq:introinfty}) can be
interpreted as the Euler-Lagrange equation for the functional
$u\mapsto \|\nabla u\|_{\infty}$, and it also coincides with the limit of the $p$-Laplace equations:
\begin{equation}
\label{eq:plaplaceintro}
\Delta_pu=|\nabla u|^{p-4}\left(|\nabla  u|^2\Delta u +(p-2)\Delta_{\infty}u\right) =0,
\end{equation}
as $p \to \infty$.
For the details of these statements and an otherwise thorough study of
(\ref{eq:introinfty}), we refer to the review work by Lindqvist
\cite{Lind_rev}. We also note that (\ref{eq:introinfty}) arises in
connection with Tug-of-War games, studied by Peres et. al in
\cite{peres}, and that it has applications within image processing as
discussed by Caselles et. al \cite{caselles}, and within glaciology \cite{Glowinski2003}.
The evolutionary counterpart $u_t= \Delta_{\infty}u$
and related equations have recently received attention, see for
example \cite{crandall2003another},  \cite{Juutinen2006}. 

The study of \eqref{eq:introinfty} is  difficult because the equation
is both fully nonlinear and   degenerate elliptic. Since it
cannot be written in divergence form, solutions are understood in the
sense of viscosity, introduced by Crandall and Lions in
\cite{crandall1983viscosity}, is required. This approach was taken in
\cite{bhattacharya1989limits} to show convergence of
(\ref{eq:plaplaceintro}) to (\ref{eq:introinfty}).  
Several other approximation methods have been used, notably Jensen's auxiliary equations:
\begin{align*}
&\min\{\Delta_{\infty} v, |\nabla v|-\epsilon\}=0, \\
&\max\{\Delta_{\infty}u, \epsilon-|\nabla u|\}=0,
\end{align*}
in the context of the comparison principle in \cite{jensen}.
Another interesting device is the ``patching solutions,'' introduced
by Crandall, Gunnarson and Wang in
\cite{crandall_gunnarsson_wang}. There is also the ``easy'' proof of
uniqueness by Armstrong and Smart \cite{Armstrong2010}.
 
In this note, we introduce a new approach. Attempting to eliminate the
domain where $\nabla u$ vanishes, we minimize the variational energies of the form:
\begin{equation} \label{eq:introvarint}
E_p^\sigma(u) = \int_{\Omega} \{|\nabla u|^2-\sigma\}_+^\frac{p}{2} \ dx,
\quad p>2, \quad \sigma >0,
\end{equation}
and let $\sigma \to 0$ and $p \to \infty.$ 
Our first main result is that, regardless what order of limits is
taken, minimizers of \eqref{eq:introvarint} converge to the viscosity
solutions of the $\infty$-Laplace equation \eqref{eq:introinfty}. 
\begin{theorem} \label{thm:thegoodstuffs}
Let $u_{p, \sigma}$ denote minimizers of \eqref{eq:introvarint}, let
$u_p$ be the solution to the $p$-Laplace equation
\eqref{eq:plaplaceintro} and $u_{\infty}$ be a solution to
\eqref{eq:introinfty}. Then the following diagram of convergence
commutes:
\[ \begin{tikzcd}
    u_{p, \sigma} \arrow{d}{\sigma \to 0} \arrow{r}{p \to \infty}  & u_{\infty, \sigma} \arrow{d}{\sigma \to 0} \\
     u_p \arrow{r}{p \to \infty} & u_{\infty}
  \end{tikzcd}. \]
The intermediate limit $u_{\infty, \sigma}$ above solves the
interesting equation:
\[ \{|\nabla u|^2-\sigma\}_+\Delta_{\infty}u =0. \]
\end{theorem}

Since our approach focuses on the convergence of minimizers of
 $p$-energies, it is natural to consider the $\Gamma$-convergence of
 the corresponding functionals. 
Indeed, our second result states that the functionals corresponding to
the minimizers in Theorem \ref{thm:thegoodstuffs}, $\Gamma$-converge
with respect to the uniform norm in $C(\overline{\Omega})$:

\begin{theorem} \label{thm:gammaconvergences}
The following diagram of $\Gamma$-convergences with respect to the
uniform norm on $C(\overline{\Omega})$ commutes:
\[ \begin{tikzcd}
    E_{p}^{\sigma} \arrow{d}{\sigma \to 0} \arrow{r}{p \to \infty}  & E_{\infty}^{\sigma} \arrow{d}{\sigma \to 0} \\
     E_p \arrow{r}{p \to \infty} & E_{\infty},
  \end{tikzcd}\]
where $E_{p}^{\sigma}$, $E_{\infty}^{\sigma}$, $E_p$ and $ E_{\infty}
$ are the functionals with minimizers as in Theorem \ref{thm:thegoodstuffs}. 
\end{theorem}

The article is structured as follows. In Section \ref{sec:varint} we
introduce the variational integral (with $\sigma =1$ for simplicity),
and prove the existence of a unique minimizer. A comparison result is
established for viscosity solutions of the corresponding
Euler-Lagrange equation. In Section \ref{sec:Jensen} we introduce
Jensen's Auxiliary Equations,  and complete the proof of Theorem
\ref{thm:thegoodstuffs}. Section \ref{sec:gamma} is dedicated to the
$\Gamma$-convergence of the corresponding functionals, and a proof of
Theorem \ref{thm:gammaconvergences}.  

\subsection{Notation}
We write:  $\{u\}_+=\max\{u, 0\}$, and by $\langle a, b \rangle$ we denote
the Euclidean inner product of vectors $a, b \in \mathbb{R}^n$. 
The set $\Omega\subset\mathbb{R}^n$ will always be open and bounded,
with the diameter:
\[ \text{diam}\, \Omega = \sup_{x, y \in \Omega}|x-y|. \] 
For integrable functions $u:\Omega\to\mathbb{R}$, we let
$u_{\Omega}=\fint_\Omega u \ dx$ stand for the average of $u$ on $\Omega$.
For $p\in [1,\infty)$, the space $W^{1, p}(\Omega)$ is the Sobolev space of functions $u$ with the weak gradient 
$\nabla u =(u_{x_1}, u_{x_2}, \ldots, u_{x_n})$, equipped with the norm,
whenever finite: 
\begin{equation} \label{eq:sobnorm}
||u||_{W^{1, p}(\Omega)}=\left(\int_{\Omega}|u|^p \ dx \right)^{\frac{1}{p}} +\left(\int_{\Omega}|\nabla u|^p \ dx \right)^{\frac{1}{p}}.
\end{equation} 
The subspace  $W^{1, p}_0(\Omega)$ of $W^{1,p}(\Omega)$ coincides with the closure of $C^{\infty}_0(\Omega)$ in the
norm \eqref{eq:sobnorm}.
The Sobolev space $W^{1, \infty}(\Omega)$ is equipped with the norm,
whenever finite:
\begin{equation*}
||u||_{W^{1, \infty}(\Omega)}=\mbox{ess.sup}_{\Omega}|u|+\mbox{ess.sup}_{\Omega}|\nabla u|.
\end{equation*} 
Finally, by $\operatorname{LSC}(\Omega)$ we denote the linear space of
lower semicontinuous functions from $\Omega$ to
$\mathbb{R}\cup\{+\infty\}$,  and $\operatorname{USC}(\Omega) = - \operatorname{LSC}(\Omega)$
denotes the space of the upper semicontinuous functions. 

\section{The truncated energy functional} \label{sec:varint}

In this section, we consider the energy functionals:
\begin{equation} \label{eq:varint}
E_p^{\sigma}(u)=\int_{\Omega} \{|\nabla u|^2-\sigma\}_+^\frac{p}{2} \ dx \quad \mbox{
  for } u\in W^{1,p}(\Omega).
\end{equation}
We prove existence of the minimizers and comparison principle and
derive the weak formulation of (\ref{eq:varint}). For the sake of the
proofs, it suffices to treat the case $\sigma=1$.

\begin{lemma} \label{exists}
Let $p>2$ and let $f \in W^{1, p}(\Omega)$. Then there exists at least
one minimizer of the energy $E_p^1$ in \eqref{eq:varint} on $\{u\in
W^{1,p}(\Omega); ~ u-f \in W_0^{1, p}(\Omega)\}$.
\end{lemma}

\begin{proof}
Observe that:
$$\int_\Omega |\nabla u|^p \ dx \leq \int_\Omega\big|
\{|\nabla u|^2-1\}_+ + 1\big|^{\frac{p}{2}} \ dx \leq C_{\Omega, p}
\big(E_p^1(u) + 1\big).$$
Existence of a minimizer to (\ref{eq:varint}) follows then by the
direct method of calculus of variations, in view of $E_p^1$ being sequentially weakly lower semicontinuous on
$W^{1,p}(\Omega)$, which is a consequence of the
classical Morrey's theorem \cite{Dacorogna} since the 
integrand density function: 
\begin{equation}\label{f}
f:\mathbb{R}^n\to [0,\infty), \qquad f(a) = \{|a|^2-1\}_+^{\frac{p}{2}}
\end{equation}
is  continuous and convex.
\end{proof}

\begin{lemma} \label{th:equivalences}
Let $p>2$ and let $u \in W^{1, p}(\Omega)$. The following conditions are equivalent:
\begin{itemize}
\item[(i)] $E^1_p(u)\leq E^1_p(w)$ for all $w\in W^{1,p}(\Omega)$ such
  that $w-u\in W^{1,p}_0(\Omega)$.
\item[(ii)] For every $v\in W^{1,p}_0(\Omega)$ there holds:
\begin{equation} \label{eq:weakeq}
\int_{\Omega} \{|\nabla u|^2-1\}_+^{\frac{p}{2}-1} \langle \nabla u,  \nabla v\rangle \ dx =0.
\end{equation}
\end{itemize}
\end{lemma}

\begin{proof}
By inspecting the convex density function $f$ in (\ref{f}), we see that:
\begin{equation}\label{conv_f}
\{|a|^2-1\}^{\frac{p}{2}}_+ \geq \{|b|^2-1\}^{\frac{p}{2}}_++
p\{|b|^2-1\}_+^{\frac{p}{2}-1}\langle b, a-b\rangle\qquad \mbox{for all }\, a,b\in\R^n. 
\end{equation}

Assume that $u$ is minimizing as in (i). Then, for every $\phi\in \mathcal{C}_0^\infty(\Omega)$
and every $|\epsilon|\leq 1$ we have, in view of (\ref{conv_f}):
$$E_p^1(u)\leq E_p^1(u+\epsilon\phi) \leq E_p^1 (u) + \epsilon
p\int_\Omega \{|\nabla u + \epsilon\nabla
\phi|^2-1\}_+^{\frac{p}{2}-1} \langle \nabla u+\epsilon \nabla \phi,
\nabla \phi\rangle \ dx, $$
which implies:
$$ (\mbox{sgn}\,\epsilon) \cdot \int_\Omega \{|\nabla u + \epsilon\nabla
\phi|^2-1\}_+^{\frac{p}{2}-1} \langle \nabla u+\epsilon \nabla \phi,
\nabla \phi\rangle \ dx \geq 0. $$
Since the  integrands above are dominated by $C(|\nabla u|^p+1)\in L^1(\Omega)$,
passing to the limit with $\epsilon\to 0$ results in (ii). The same
argument can then be extended to all test
functions $v\in W^{1,p}_0(\Omega)$, by density and because $\{|\nabla u|^2-1\}_+^{\frac{p}{2}-1}
\nabla u\in L^{\frac{p}{p-1}}(\Omega)$. 

Conversely, if (ii) holds, then (\ref{conv_f}) yields for every $w$ as in (i):
\begin{equation*}
0 = p \int_{\Omega} \{|\nabla u|^2-1\}_+^{\frac{p}{2}-1} \langle
\nabla u, \nabla w - \nabla u\rangle \ dx \leq E_p^1(w) - E^1_p(u).
\end{equation*}
This concludes the proof.
\end{proof}

We further have the following comparison principle:

\begin{theorem} \label{thm:Comparison}
Let $p>2$ and assume that $u_1, u_2 \in W^{1,p}(\Omega)$ satisfy:
\begin{equation} \label{super_sub_weak}
\begin{split}
\int_{\Omega} \{|\nabla u_1|^2-1\}_+^{\frac{p}{2}-1} \langle \nabla u_1,  \nabla v\rangle \ dx \leq 0,
\qquad
& \int_{\Omega} \{|\nabla u_2|^2-1\}_+^{\frac{p}{2}-1} \langle \nabla
u_2,  \nabla v\rangle \ dx \geq 0, \\ & 
\qquad \mbox{for all } \,  v\in W_0^{1,p}(\Omega), \quad v\geq 0 \mbox{   a.e. in } \Omega.
\end{split}
\end{equation}
If $(u_1-u_2)_+\in W_0^{1,p}(\Omega)$, then $u_1\leq u_2$ a.e. in the set: 
$$\{x\in \Omega; ~ |\nabla u_1(x)|>1\}\cup \{x\in \Omega; ~ |\nabla u_2(x)|>1\}.$$
\end{theorem}
\begin{proof}
{\bf 1.}  We first show that:
\begin{equation}\label{inproof3}
\nabla u_1 = \nabla u_2 \qquad \mbox{a.e. in } \ \Omega\setminus A,
\end{equation}
where we define:
$$A= \{x\in\Omega; ~ |\nabla  u_1(x)|\leq 1\}\cap \{x\in \Omega;~
|\nabla  u_2(x)|\leq 1\}.$$
Applying (\ref{eq:weakeq}) to $v=(u_1-u_2)_+$, it follows that:
\begin{equation}\label{inproof}
\begin{split}
0 & \geq  \int_\Omega \langle \{|\nabla u_1|^2-1\}_+^{\frac{p}{2}-1}\nabla
u_1 - \{|\nabla u_2|^2-1\}_+^{\frac{p}{2}-1}\nabla u_2, \nabla
(u_1-u_2)_+\rangle \ dx \\ & = \int_{\{u_1>u_2\}\setminus A} \langle \{|\nabla u_1|^2-1\}_+^{\frac{p}{2}-1}\nabla
u_1 - \{|\nabla u_2|^2-1\}_+^{\frac{p}{2}-1}\nabla u_2, \nabla u_1-\nabla u_2 \rangle \ dx, 
\end{split}
\end{equation}

For any $a,b\in\R^n$, consider now the expression:
$$\varphi_{a,b} = \langle \{|a|^2-1\}_+^{\frac{p}{2}-1} a - \{|b|^2-1\}_+^{\frac{p}{2}-1}b, a-b \rangle. $$
We claim that $\varphi_{a,b}\geq 0$ and that if $|a|>1$ or $|b|>1$
then $\varphi_{a,b}=0$ implies $a=b$. Indeed, when $|a|>1$ and
$|b|\leq 1$, then $\langle a, a-b\rangle = |a|^2 - \langle a, b\rangle
\geq |a|^2 - |a| |b|>0$. On the other hand, for $|a|, |b|> 1$, a
direct calculation yields:
\begin{equation*}
\begin{split}
\varphi_{a,b} & = \langle \big||a|^2-1\big|^{\frac{p}{2}-1} a -
\big||b|^2-1\big|^{\frac{p}{2}-1}b, a-b \rangle \\ & =
\frac{1}{2}\Big( \big||a|^2-1\big|^{\frac{p}{2}-1} + \big||b|^2-1\big|^{\frac{p}{2}-1}\Big) |a-b|^2 +  
\frac{1}{2}\Big( \big||a|^2-1\big|^{\frac{p}{2}-1} -
\big||b|^2-1\big|^{\frac{p}{2}-1}\Big) \big(|a|^2 - |b|^2\big). 
\end{split}
\end{equation*}
The right hand side above is nonnegative and it equals $0$
only when both terms are null, hence when $a=b$. This concludes the
proof of (\ref{inproof3}) by (\ref{inproof}).

\medskip

{\bf 2.} Observe now that for every $c\in\mathbb{R}$, the following
are nonnegative admissible test functions for (\ref{eq:weakeq}):
\begin{equation}\label{WC}
\begin{split}
& v_c = \big\{\min\{u_1, c\} - u_2\big\}_+ = \min\big\{\{u_1-u_2\}_+,
\{c-u_2\}_+\big\} \in W_0^{1,p}(\Omega), \\ &
\bar v_c = \big\{u_1 - \max\{u_2, c\} \big\}_+ = \min\big\{\{u_1-u_2\}_+, \{u_1-c\}_+\big\} \in W_0^{1,p}(\Omega),
\end{split}
\end{equation}
To this end, let $\{\phi^n\}_{n=1}^\infty$ be a sequence in
$\mathcal{C}_0^\infty(\Omega)$, converging to $(u_1- u_2)_+$ 
in $W^{1,p}(\Omega)$. Then $v_c^n = \min\big\{\phi^n, \{c-u_2\}_+\big\}$ and 
$\bar v_c^n = \min\big\{\phi^n, \{u_1-c\}_+\big\}$ belong to 
$W_0^{1,p}(\Omega)$, because:
$$(\mbox{supp}\, v_c^n)\cup (\mbox{supp}\, v_c^n) \subset
\mbox{supp}\, \phi^n\subset \Omega.$$ 
Since the operations $\min /
\max$ are continuous with respect to the $W^{1,p}(\Omega)$ norm, it
follows that $v_c^n\to v_c$ and $\bar v_c^n\to \bar v_c$ in $W^{1,p}(\Omega)$, proving (\ref{WC}).
Further, condition (\ref{inproof3}) implies that:
$$\nabla v_c = \left\{\begin{array}{ll} \nabla u_1 - \nabla u_2 &  \mbox{ when } c\geq u_1>u_2\\
- \nabla u_2 &  \mbox{ when } u_1> c > u_2\\
0 &  \mbox{ otherwise} \end{array} \right. = 
\left\{\begin{array}{ll} - \nabla u_2 &  \mbox{ when } u_1> c > u_2\\
0 &  \mbox{ otherwise} \end{array} \right.$$
together with:
$$\nabla \bar v_c = \left\{\begin{array}{ll} \nabla u_1 - \nabla u_2 &
    \mbox{ when } u_1>u_2\geq c\\
\nabla u_1 &  \mbox{ when } u_1> c > u_2\\
0 &  \mbox{ otherwise} \end{array} \right. = 
\left\{\begin{array}{ll}  \nabla u_1 &  \mbox{ when } u_1> c > u_2\\
0 &  \mbox{ otherwise.} \end{array} \right.$$

We may thus use the assumption (\ref{super_sub_weak}) to get: 
\begin{equation}\label{inproof2}
\begin{split}
0 & \leq \int_\Omega \{|\nabla u_2|^2-1\}_+^{\frac{p}{2}-1}\langle \nabla u_2, \nabla
v_c\rangle \ dx = - \int_{\{u_1>c>u_2\}\cap \{|\nabla
  u_2|>1\}}  \{|\nabla u_2|^2-1\}_+^{\frac{p}{2}-1} |\nabla
u_2|^2  \ dx, \\
0 & \geq \int_\Omega \{|\nabla u_1|^2-1\}_+^{\frac{p}{2}-1}\langle\nabla u_1, \nabla
\bar v_c\rangle \ dx =  \int_{\{u_1>c>u_2\}\cap \{|\nabla
  u_1|>1\}}  \{|\nabla u_1|^2-1\}_+^{\frac{p}{2}-1}|\nabla u_1|^2 \ dx.
\end{split}
\end{equation}
Since both integrands in the right hand sides above are strictly
positive, we conclude that there must be:
$$ \big| \{x\in\Omega; ~ u_1(x)<c<u_2(x)\}\cap (\Omega\setminus
A)\big| = 0 \qquad \mbox{ for all } \, c\in\mathbb{R},$$
achieving the proof.
\end{proof}

\begin{corollary}
Let $p>2$ and let $u_1, u_2$ be two minimizers of $E_p^1$ on
$\{u\in W^{1,p}(\Omega); ~ u-f \in W_0^{1, p}(\Omega)\}$, for some 
$f \in W^{1, p}(\Omega)$. Then $ u_1=u_2$ a.e. in the set:
$$\{x\in \Omega; ~ |\nabla  u_1(x) |>1\}\cup\{x\in \Omega; ~ |\nabla  u_2|>1\}.$$
\end{corollary}

\begin{proof}
Since $u_1,u_2$ satisfy both conditions in (\ref{super_sub_weak}) by Lemma
\ref{th:equivalences} and also $(u_1-u_2)_+, (u_2-u_1)_+\in W_0^{1,p}(\Omega)$, Theorem
\ref{thm:Comparison} implies the claim.
\end{proof}

\medskip

We close this section by noting that the Euler-Lagrange equations of (\ref{eq:varint}) are
obtained by integrating by parts in its weak formulation:
\begin{equation}\label{2.33}
\int_\Omega \{|\nabla u|^2-\sigma\}_+^{\frac{p}{2}-1}\langle \nabla u,
\nabla v\rangle \ dx = 0 \qquad \mbox{for all } \ v\in W_0^{1,p}(\Omega)
\end{equation}
namely:
\begin{equation} \label{eq:strongeq}
L_p^\sigma u:=\{|\nabla u|^2-\sigma\}_+^{\frac{p-4}{2}}\Big(\{|\nabla
u|^2-\sigma\}_+\Delta u + (p-2) \Delta_{\infty}u\Big) =0. 
\end{equation}
where $\Delta_\infty$ is defined as in (\ref{eq:introinfty}) and with the
convention that $\{a\}_+^{\frac{p-4}{2}}=0$ whenever $a\leq 0$.

\section{Two limit procedures}

\begin{definition}\label{sub_super}
We say that $u \in \operatorname{LSC}(\Omega)$ that is not
equivalently  $+\infty$ in $\Omega$, is a viscosity supersolution of
\eqref{eq:strongeq} if, whenever $u-\phi $ has a local minimum at some
$x_0\in\Omega$ and for some $\phi\in\mathcal{C}^\infty(\Omega)$, it results in:
\begin{align*}
L_p^\sigma \phi(x_0)\leq 0.
\end{align*}
Also, $u \in \operatorname{USC}(\Omega)$  that is not
equivalently $-\infty$ in $\Omega$,
is a viscosity subsolution of \eqref{eq:strongeq} if $(-u)$ is a viscosity supersolution.
A continuous function $u:\Omega\to\R$ is called a viscosity solution
if it is both a viscosity sub- and supersolution.
\end{definition}

\begin{theorem}\label{th_cont}
If $u\in W^{1,p}(\Omega)$ is a continuous function satisfying
\eqref{2.33}, then it is a viscosity solution to \eqref{eq:strongeq}.
\end{theorem}
\begin{proof}
Without loss of generality, it suffices to treat the case $\sigma=1$.
For a fixed $x_0 \in \Omega$, consider a supporting test function
$\phi\in\mathcal{C}^\infty(\Omega)$ such that $\phi(x_0) = u(x_0)$ and
$\phi\leq u$ in $\Omega$. By possibly modifying $\phi$ to $\phi -
|x-x_0|^2$, we may without loss of generality assume that $\phi(x) <
u(x)$ for all $x\in \Omega\setminus \{x_0\}$. 
If, by contradiction, $L_p^1\phi(x_0)>0$, then there must be for
some $r>0$:
$$|\nabla \phi(x)|>1 \, \mbox{ and } \, L_p^1\phi(x)>0 \quad \mbox{ for all }
x\in \bar B_r(x_0)\subset \Omega, $$
Let $m_r=\min_{\partial B_r(x_0)}(u-\phi)>0$ and consider the
modification $\bar\phi = \phi + m_r$. The condition $L_p\bar\phi > 0$ in
$B_r(x_0)$ is equivalent to:
$$\int_{B_r(x_0)} \{|\nabla \bar\phi|^2-1\}_+^{\frac{p}{2}-1} \langle
\nabla \bar\phi,  \nabla v\rangle \ dx \leq 0 \quad \mbox{ for all }
v\in W^{1,p}_0(\Omega), \,\, v\geq 0 \, \mbox{ a.e. in } \Omega.$$
Since $\bar\phi\leq u$ on $\partial B_r(x_0)$, the comparison
principle in Theorem \ref{thm:Comparison} implies $\bar{\phi}\leq u$ in
$B_r(x_0)$, contradicting $\phi(x_0) = u(x_0)$. This proves that $u$
is a viscosity supersolution to (\ref{eq:strongeq}). The proof of $u$
being a viscosity subsolution follows in the same manner.
\end{proof}

\medskip

We now show that the limit equation of \eqref{eq:strongeq} as $p \to \infty$ is:
\begin{equation} \label{eq:strangeinf}
\{|\nabla u|^2-\sigma|\}_+\Delta_{\infty}u =0.
\end{equation}
The viscosity sub- and supersolutions to  (\ref{eq:strangeinf})
(respectively, $\Delta_pu = 0$) are defined as in Definition
\ref{sub_super}, with the operator $L_p^\sigma$ replaced by:
$$L_\infty^\sigma u = \{|\nabla u|^2-\sigma|\}_+\Delta_{\infty}u$$
(respectively, with the operator $\Delta_p$ in (\ref{eq:plaplaceintro})).

\medskip

\begin{theorem} \label{thm:ptoinfty}
For every $\sigma >0$ and every $p>n$, let $u_p^\sigma\in
W^{1,p}(\Omega)$ be a solution to (\ref{2.33}) sa\-tis\-fy\-ing
$u_p^\sigma-f\in W_0^{1,p}(\Omega)$, where $f\in W^{1,\infty}(\Omega)$
is a given function. Then we have:
\begin{itemize}
\item[(i)] As $\sigma\to 0$, any sequence of $\{u_p^\sigma\}_\sigma$
  converges weakly in $W^{1,p}(\Omega)$, and in
  $\mathcal{C}_{loc}^\alpha(\Omega)$ for every $\alpha\in (0,
  1-\frac{n}{q})$, to the unique function $u_p$ that is a viscosity
  solution to $\Delta_p u = 0$ in $\Omega$ satisfying $u_p-f\in W_0^{1,p}(\Omega)$.
\item[(ii)] As $p\to\infty$, every sequence of $\{u_p^\sigma\}_p$
  contains a subsequence $\{u^\sigma_{p_j}\}_j$ converging
weakly in $W^{1,q}(\Omega)$ for every $q\in (1, \infty)$ and in $\mathcal{C}^\alpha_{loc}(\Omega)$ for
every $\alpha\in (0,1)$, to a viscosity solution $u^\sigma_{\infty}$ of (\ref{eq:strangeinf}). 
\end{itemize}
\end{theorem}
\begin{proof}
{\bf 1.} For a fixed $\sigma>0$ and $q>n$, observe the following uniform bound, valid
for all $p\geq q$:
\begin{equation*}
\begin{split}
\big(\int_\Omega |\nabla u_p^\sigma|^q \ dx\big)^{\frac{2}{q}} & \leq
\big(\int_\Omega\big| \{|\nabla
u_p^\sigma|^2-\sigma\}_++\sigma\big|^{\frac{q}{2}}\ dx\big)^{\frac{2}{q}} \leq 
\big(\int_\Omega \{|\nabla
u_p^\sigma|^2-\sigma\}_+^{\frac{q}{2}}\ dx\big)^{\frac{2}{q}} + C_{\Omega, q}\sigma
\\ & \leq C_{\Omega, q} \big(E_p^\sigma(u_p^\sigma)^{\frac{2}{p}} +\sigma\big) \leq
C_{\Omega, q} \big(E_p^\sigma(f)^{\frac{2}{p}} +\sigma\big) \leq
C_{\Omega, q}\big(\|\nabla f\|^2_{L^\infty(\Omega)} + \sigma\big) .
\end{split}
\end{equation*}
Moreover: 
\begin{equation*}
\begin{split}
\|u_p^\sigma\|_{L^q(\Omega)}^2 \leq C_{\Omega, q}\big(\|\nabla
u_p^\sigma - \nabla f\|^2_{L^q(\Omega)} + \|f\|^2_{L^q(\Omega)}\big)
\leq C_{\Omega, q}\big(\| f\|^2_{W^{1,\infty}(\Omega)} +\sigma\big),
\end{split}
\end{equation*}
so we may conclude that the family $\{u_p^\sigma\}_{p\geq q, ~
  \sigma\in\{0, \sigma_0)}$ is uniformly bounded in $W^{1,q}(\Omega)$.
We also recall the Morrey inequality, which yields for all $p\geq q$
and $\sigma>0$:
\begin{equation}\label{mor}
|u_p^\sigma(x_1) - u_p^\sigma(x_2)| \leq C_{n,q} |x_1-x_2|^{1-\frac{n}{q}} \|\nabla
u^\sigma_p\|_{L^q(B_r(x))} \qquad \mbox{for all } x_1,x_2\in \bar B_r(x) \subset\Omega,
\end{equation}
Finally, if $\phi\in\mathcal{C}^\infty(\Omega)$ and if $x_p^\sigma$ is
a local minimum of $u_p^\sigma-\phi$ in $\Omega$, then Theorem
\ref{th_cont} implies that:
\begin{equation*}
\{|\nabla\phi(x_{p}^\sigma)|^2-\sigma\}_+^{\frac{p-4}{2}}\left(\{|\nabla\phi(x_{p}^\sigma)|^2-\sigma\}_+\Delta
  \phi(x_{p}^\sigma) +(p-2)\Delta_{\infty}\phi(x_{p}^\sigma)\right) \leq 0,
\end{equation*}
which is equivalent to:
\begin{equation}\label{mor2}
|\nabla\phi(x_{p}^\sigma)|^2\leq \sigma \quad \mbox{ or }\quad 
(|\nabla\phi(x_{p}^\sigma)|^2-\sigma )\Delta
  \phi(x_{p}^\sigma) + (p-2)\Delta_{\infty}\phi(x_{p}^\sigma) \leq 0.
\end{equation}

\medskip

{\bf 2.} We are now in a position to prove (i). Let
$\{u_p^{\sigma_j}\}_{\sigma_j\to 0}$ be a subsequence that converges
weakly in $W^{1,p}(\Omega)$ and pointwise a.e. in $\Omega$, to some
$u_p\in W^{1,p}(\Omega)$ such that $u_p - f\in W^{1,p}_0(\Omega)$.
In view of (\ref{mor}), the said sequence is locally equibounded in $\mathcal{C}^{1-\frac{n}{q}}(\Omega)$,
so the Ascoli-Arzel\`a theorem yields its convergence (possibly up to another
subsequence), locally in $\mathcal{C}^\alpha(\Omega)$ with $\alpha\in
(0, 1-\frac{n}{q})$, to $u_p$. We now show that the limit $u_p$ is a
viscosity solution to $\Delta_pu=0$.

Let $x_0\in \Omega$ and let $\phi\in
\mathcal{C}^\infty(\Omega)$ be such that $\phi(x_0) = u_p(x_0)$
and $\phi(x)<u_p(x)$ for all $x\in \Omega\setminus \{x_0\}$.
Firstly, the local uniform convergence of $\{u_{p}^{\sigma_j}\}_j$ 
implies that the minima $x_p^{\sigma_j}$ of $u_p^{\sigma_j}-\phi$ on $\bar B_r(x_0)\subset\Omega$
must converge to $x_0$ as $\sigma_j\to\infty$. If $\nabla \phi(x_0)=0$
then automatically $\Delta_p\phi(x_0)=0$. Otherwise, it follows that
for all $j$ sufficiently large, the second condition in (\ref{mor2})
must be valid. Passing with $\sigma_j\to 0$, we obtain
$\Delta_p\phi(x_0)\leq 0$. This proves that $u_p$ is a viscosity
supersolution to $\Delta_pu=0$. The proof of subsolution follows in
the same manner.  

Since viscosity solutions are unique, as they coincide with the
minimizers of $\int_\Omega |\nabla u|^p \ dx$ with $u-f\in
W^{1,p}(\Omega)$, the limit $u_p$ is independent of the subsequence
$\{u_p^{\sigma_j}\}_j$. This concludes the proof of (i). 

\medskip

{\bf 3.} To prove (ii), we observe that, for a fixed $\sigma>0$, the
tail of any subsequence of $\{u_p^\sigma\}_\sigma$ is uniformly
bounded in every $W^{1,q}(\Omega)$ for $q\in (n, +\infty)$. By a
diagonal procedure, we may extract a subsequence
$\{u_{p_j}^\sigma\}_{p_j\to ^\infty}$ converging weakly in every
$W^{1,q}(\Omega)$ to some function $u^\sigma_\infty\in W^{1,q}(\Omega)$.
Without loss of generality, $\{u_{p_j}^\sigma\}_j$ converges also pointwise
a.e. in $\Omega$. In view of (\ref{mor}), the said sequence is locally equibounded in
$\mathcal{C}^\alpha(\Omega)$ for every $\alpha\in (0,1)$, which yields
convergence of (possibly
another subsequence of) $\{u^\sigma_{p_j}\}_j$, locally in each
$\mathcal{C}^\alpha(\Omega)$, to $u^\sigma_\infty$.

To show that $u^\sigma_\infty$ is a viscosity supersolution to
\eqref{eq:strangeinf}, let $x_0\in \Omega$ and let $\phi\in
\mathcal{C}^\infty(\Omega)$ be such that $\phi(x_0) = u_\infty(x_0)$
and $\phi(x)<u_\infty(x)$ for all $x\in \Omega\setminus \{x_0\}$. We
will deduce that:
\begin{equation} \label{gr}
\{|\nabla\phi(x_0)|^2-\sigma\}_+\Delta_{\infty}\phi(x_0) \leq 0.
\end{equation}
As before, the local uniform convergence 
implies that the minima $x_{p_j}^\sigma$ of $u_{p_j}^\sigma-\phi$ on $\bar B_r(x_0)\subset\Omega$
must converge to $x_0$ as $p_j\to\infty$. If $|\nabla \phi(x_{p_j}^\sigma)|\leq \sigma$ for infinitely many indices $j$,
then $|\nabla \phi(x_0)|\leq \sigma$, yielding (\ref{gr}). On the other hand,
if $|\nabla \phi(x_{p_j}^\sigma)|>\sigma$ for all $j$ sufficiently
large, then dividing by $p_j-2$ in the second condition in
(\ref{mor2}) and passing to the limit with $p_j\to \infty$ we obtain $\Delta_{\infty}\phi(x_0) \leq
0$, again resulting in (\ref{gr}). This ends the proof of $u_\infty ^\sigma$
being a viscosity supersolution to \eqref{eq:strangeinf}. 
The proof for subsolution follows in the same manner.
\end{proof}

\section{Jensen's Auxiliary Equations}
\label{sec:Jensen}
In this section, we prove that as $\sigma \to \infty$, the viscosity solutions of \eqref{eq:strangeinf} converges to the viscosity solutions of the $\infty$-Laplace equation $\Delta_{\infty}u=0$. Jensen's Auxiliary equations are employed to control the convergence.

We consider the variational integral
\begin{equation}
\label{eq:varintsigma}
\int_{\Omega}\frac{1}{p}\{|\nabla  u|^2-\sigma^2\}^{\frac{p}{2}}_+ -\sigma^{p-4}u \ dx,
\end{equation}
where the admissible functions $u$ are so that $u-f \in W^{1, p}_0(\Omega)$ for a fixed function $f$. We have the following. 
\begin{lemma}
	The variational integral \eqref{eq:varintsigma} has a unique minimizer.
\end{lemma}
\begin{proof}
The existence is a direct result of the direct method of calculus of variations. The proof of uniqueness is similar to the proof shown in Theorem \ref{thm:Comparison}.
\end{proof}
\begin{theorem}
	The minimizer of the variational integral \eqref{eq:varintsigma} has weak Euler-Lagrange equation
	\begin{equation}
	\label{eq:jensen_euler-lagrange}
	\int_{\Omega}\{|\nabla  u|^2-\sigma^2\}_+^{\frac{p}{2}-1}\langle \nabla  u, \nabla\phi\rangle -\sigma^{p-4}\phi \ dx =0,
	\end{equation}
	for every smooth $\phi$.
\end{theorem}
\begin{proof}
	For $\epsilon$, define the function
	\[
	G(\epsilon)=\int_{\Omega}\frac{1}{p}\{|\nabla  u+\epsilon\phi|^2-\sigma^2\}_+^{\frac{p}{2}} -\sigma^{p-4}(u+\epsilon\phi) \ dx.
	\]
	Differentiating with respect to $\epsilon$ yields
	\[
	G'(\epsilon)=\int_{\Omega}\frac{1}{p}\frac{p}{2}\{|\nabla  u+\epsilon\phi|^2-\sigma^2\}_+^{\frac{p}{2}-1}2\langle \nabla  u, \nabla(u+\epsilon \phi)\rangle -\sigma^{p-4}\phi \ dx.
	\]
	Since $u$ was assumed to be minimizing, we must have that $G$ has a minimum at $\epsilon =0$, that is $G'(0)=0$. This is precisely \eqref{eq:jensen_euler-lagrange}.
\end{proof}

With suitable conditions on $u$, we can integrate \eqref{eq:jensen_euler-lagrange} by parts to give the strong equation
\begin{equation}
\label{eq:strongjensen}
\{|\nabla  u|^2-\sigma\}_+^{\frac{p-4}{2}}(\{|\nabla  u|^2-\sigma^2\}\Delta u +(p-2)\Delta_{\infty}u)=-\sigma^{p-4}.
\end{equation}
Note that the left hand side of the above is strictly negative, and this prohibits $\{|\nabla  u|^2-\sigma^2\}_+ =0$, and so, indeed,
\begin{equation}
\label{eq:gradientrestriction}
|\nabla  u|>\sigma.
\end{equation}
We aim at deriving the limit of \eqref{eq:strongjensen} as $p \to \infty$.
\begin{theorem}
	As $p \to \infty$ the viscosity solution of \eqref{eq:strongjensen} converge to the viscosity solution of Jensen's auxilliary equation
	\begin{equation}
	\label{eq:jensenaux1}
	\max\{2\sigma^2-|\nabla  u|^2, \Delta_{\infty}u\}=0.
	\end{equation}
\end{theorem}
\begin{proof}
Standard compactness arguments, cf Theorem \ref{thm:ptoinfty}, yields the existence of a subsequence $\{u_{p_j}\}_j$ converging locally uniformly to a continuous function $u_{\infty}$. 

We first prove that $u_{\infty}$ is a viscosity subsolution. Let $u_{\infty}-\phi$ have a minimum at $x_{\infty}$. Since $u_{p_j}$ converges locally uniformly to $u_{\infty}$, we see that minima $x_{p_j}$ of $u_{p_j}-\phi$ converges uniformly to $x_{\infty}$. By the definition of viscosity subsolutions,
	\[
	\{|\nabla\phi(x_{p_j})|^2-\sigma^2\}_+^{\frac{p_j-4}{2}}(\{|\nabla\phi(x_{p_j})|^2-\sigma^2\}_+\Delta \phi(x_{p_j}) +(p_j-2)\Delta_{\infty}\phi(x_{p_j})\geq-\sigma^{p_j-4}.
	\]
	Since the left hand side is non-negative, cf. \eqref{eq:gradientrestriction}, we can divide through by $(\nabla\phi(x_{p_j})|^2-\sigma^2)^{\frac{p_j-4}{2}}(p_j-2)>0$  and rearrange to get
	\[
	\Delta_{\infty}\phi(x_{p_j})\geq -\frac{1}{p_j-2}\left(\frac{\sigma^2}{|\nabla\phi(x_{p_j})|^2-\sigma^2}\right)^{\frac{p_j-4}{2}}-(|\nabla\phi(x_{p_j})|^2-\sigma^2)\frac{\Delta\phi(x_{p_j})}{p_j-2}.
	\]
	If 
	\[
	\frac{\sigma^2}{|\nabla\phi(x_{p_j})|^2-\sigma^2}\leq 1,
	\]
	for infinitely many $j$, we see that the right hand side of the above converges to 0, and we see that $\Delta_{\infty}\phi(x_{\infty})\geq 0$ by continuity. 
	
	On the other hand, If 
	\begin{align*}
	\frac{\sigma^2}{|\nabla\phi(x_p)|^2-\sigma^2}> 1 
	\iff 2\sigma^2 -|\nabla\phi(x_p)|^2> 0,
	\end{align*}
	we see that the right hand side converges to $-\infty$, and the inequality is vacuously true. This implies that
	\[
	\max\{2\sigma^2-|\nabla\phi(x_{\infty})|^2, \Delta_{\infty}\phi(x_{\infty})\}\geq 0
	\]
	
	We now turn to supersolutions. Let $u_{\infty}-\phi$ have a maximum at $x_{\infty}$. This implies that $u_{p_j}-\phi$ has maxima at $x_{p_j}\to x_{\infty}$. We get
	\[
	\{|\nabla\phi(x_{p_j})|^2-\sigma^2\}_+^{\frac{p_j-4}{2}}(\{|\nabla\phi(x_{p_j})|^2-\sigma^2\}\Delta \phi(x_{p_j}) +(p_j-2)\Delta_{\infty}\phi(x_{p_j})\leq-\sigma^{p_j-4}.
	\]
	Arguing as above, we divide through to get
	\[
	\Delta_{\infty}\phi(x_{p_j})\leq -\frac{1}{p_j-2}\left(\frac{\sigma^2}{|\nabla\phi(x_{p_j})|^2-\sigma}\right)^{\frac{p_j-4}{2}}-(|\nabla\phi(x_{p_j})|-\sigma^2)\frac{\Delta\phi(x_{p_j})}{p_j-2}.
	\]
	If
	\[
	\frac{\sigma^2}{|\nabla\phi(x_{p_j})|^2-\sigma^2}\geq 1,
	\]
	We see that the right hand side of the above converge to $-\infty$. This is prohibited, as the left hand side $\Delta_{\infty}\phi(x_{p_j})$ remains bounded, and we get that $2\sigma^2 -|\nabla\phi(x_{p_j})|^2\leq 0$ in this case.
	
	On the other hand, if 
	\[
	\frac{\sigma^2}{|\nabla\phi(x_{p_j})|^2-\sigma^2}\leq 1,
	\]
	we get in the limit
	\[
	\Delta_{\infty}\phi(x_{\infty}) \leq 0,
	\]
	and so
	\[
	\max\{2\sigma^2 - |\nabla\phi(x_{\infty})|^2, \Delta_{\infty}\phi(x_{\infty})\}\leq 0,
	\]
	and hence $u_{\infty}$ is a viscosity solution of \eqref{eq:jensenaux1}.
\end{proof}
For the Lower Equation, the various stages are
\begin{align}
&\int_{\Omega}\frac{1}{p}\{|\nabla  u|^2-\sigma^2\}^{\frac{p}{2}}_+ +\sigma^{p-4}u \ dx, \nonumber \\
&\{|\nabla  u|^2-\sigma^2\}_+^{\frac{p-4}{2}}(\{|\nabla  u|^2-\sigma^2\}\Delta u +(p-2)\Delta_{\infty}u)=\sigma^{p-4}, \label{eq:p_lower_jensen} \\
&\min\{|\nabla  u^-|-2\sigma^2, \Delta_{\infty}u^-\}=0. \nonumber
\end{align}

We know from Theorem \ref{thm:Comparison} that viscosity solutions $u_p$ of $L^{\sigma}_pu=0$ enjoy a comparison principle. This implies the following result.
\begin{theorem}
	Let $u_p^+$ be a viscosity solution of the Upper Equation \eqref{eq:strongjensen} and $u_p^-$ be a viscosity solution of the Lower Equation \eqref{eq:p_lower_jensen}, and let $u_p$ be the viscosity solution of \eqref{eq:strongeq},  all with the same boundary values $f$. Then
	\begin{equation}
	u_p^-\leq u_p \leq u_p^+
	\end{equation}
	in the set $\Omega \setminus A_{\sigma}=\{x \in \Omega \ : \ |\nabla  u_p|>\sigma\}$.
\end{theorem}
The dead cores for the Lower and Upper equation does not count, as we have $|\nabla  u_p^-|, \ |\nabla  u_p^+|>\sigma$.

Choose a subsequence so that all three converge, say
\[
u_p^- \to u^-, \ u_p \to u, \ u_p^+\to u^-.
\]
Here $u$ is a viscosity solution of equation \eqref{eq:strangeinf}.
We ignore, for the moment, the dependence upon $\sigma$ in the  notation.
Subtracting equation \eqref{eq:p_lower_jensen} from \eqref{eq:strongjensen}, and choosing $\phi=(u_p^+-u_p^-)$ as our test function, we get
\begin{equation}
\label{eq:difference}
\begin{split}
\int_{\Omega}\big\langle\{|\nabla  u_p^+|^2-\sigma^2\}_+^{\frac{p}{2}-1} \nabla  u_p^+-\{|\nabla  u_p^-|^2-\sigma^2\}_+^{\frac{p}{2}-1} \nabla  u_p^-, \nabla(u_p^+-u_p^-)\big\rangle \ dx \\
=  \sigma^{p-4}\int_{\Omega} (u_p^+-u_p^-) \ dx.
\end{split}
\end{equation}
The expression
\begin{align*}
\langle(|b|^2-\sigma^2)^{\frac{p}{2}-1}b-(|a|^2-\sigma^2)^{\frac{p}{2}-1}a, b-a \rangle  
\end{align*}
for vectors $a$ and $b$ appears.
A calculation shows that
\begin{align*}
\langle(|b|^2-\sigma^2)^{\frac{p}{2}-1}b-(|a|^2-\sigma^2)^{\frac{p}{2}-1}a, b-a \rangle = \\
\frac{1}{2}\Big( \big||b|^2-\sigma^2\big|^{\frac{p}{2}-1} + \big||a|^2-\sigma^2\big|^{\frac{p}{2}-1}\Big) |a-b|^2 +  \\
\frac{1}{2}\Big( \big||b|^2-\sigma^2\big|^{\frac{p}{2}-1} -
\big||a|^2-\sigma^2\big|^{\frac{p}{2}-1}\Big) \big(|b|^2 - |a|^2\big),
\end{align*}
and hence
\begin{align*}
\langle(|b|^2-\sigma^2)^{\frac{p}{2}-1}b-(|a|^2-\sigma^2)^{\frac{p}{2}-1}a, b-a \rangle \geq \\
\frac{1}{2}\Big( \big||b|^2-\sigma^2\big|^{\frac{p}{2}-1} + \big||a|^2-\sigma^2\big|^{\frac{p}{2}-1}\Big) |a-b|^2 .
\end{align*}
Using convexity, we get
\begin{align*}
\frac{1}{2}\Big( \big||b|^2-\sigma^2\big|^{\frac{p}{2}-1} + \big||a|^2-\sigma^2\big|^{\frac{p}{2}-1}\Big) |a-b|^2 \\
\geq \left|\frac{(|a|^2-\sigma^2)+(|b|^2-\sigma^2)}{2}\right|^{\frac{p}{2}-1}|a-b|^2 \\
= \left|\frac{|a|^2+|b|^2}{2}-\sigma^2\right|^{\frac{p}{2}-1}|a-b|^2 \\
\geq 4\sigma^2\left|\left|\frac{a+b}{2}\right|^2-\sigma^2\right|^{\frac{p}{2}-1}.
\end{align*}
Inserting this into \eqref{eq:difference}, with $\nabla  u_p^+$, $\nabla  u_p^-$ and taking $\frac{1}{p/2-1}$ roots, we get
\begin{align*}
\left(4\sigma^2\int_{\Omega}\left|\left|\frac{\nabla  u_p^+-\nabla  u_p^-}{2}\right|^2-\sigma^2\right|^{\frac{p}{2}-1} \ dx\right)^\frac{1}{p/2-1}\leq \left(\sigma^{p-4}\int_{\Omega}u_p^+-u_p^- \ dx\right)^{\frac{1}{p/2-1}},
\end{align*}
and upon letting $p \to \infty$:
\begin{equation*}
 \mbox{ess.sup}_{\Omega}(|\nabla  u^++\nabla  u^-|^2-4\sigma^2) \leq C\sigma^2, 
\end{equation*}
or
\[
\mbox{ess.sup}_{\Omega}(|\nabla  u^+-\nabla  u^-|) \leq \mbox{ess.sup}_{\Omega}(|\nabla  u^++\nabla  u^-|) \leq C'\sigma,
\]
Integrating, this is 
\begin{equation}
\label{eq:jensenconvergence}
\mbox{ess.sup}_{\Omega}|u^+-u^-| \leq \sigma C'\text{diam}(\Omega)
\end{equation}
or 
\[
u_{\sigma}^-\leq u_{\sigma} \leq u_{\sigma}^+ \leq u_{\sigma}^-+C''\sigma.
\]
This proves that as $\sigma \to 0$, the viscosity solutions of \eqref{eq:strangeinf} converge to infinity harmonic functions, since solutions of Jensen's Upper and Lower equation have this property.

We know from \cite{bhattacharya1989limits} and \cite{jensen} that as
$p\to \infty$ viscosity solutions $u_p$ of the
$p$-Laplace equation converge uniformly to the $\infty$-harmonic
function. This, combined with Theorem \ref{thm:ptoinfty} and the estimate \eqref{eq:jensenconvergence}, implies the following diagram of convergence 
\[
  \begin{tikzcd}
    u_{p, \sigma} \arrow{d}{\sigma \to 0} \arrow{r}{p \to \infty}  & u_{\sigma} \arrow{d}{\sigma \to 0} \\
     u_p \arrow{r}{p \to \infty} & u_{\infty},
  \end{tikzcd}
\]
proving Theorem \ref{thm:thegoodstuffs}.

\section{$\Gamma$-convergence}
\label{sec:gamma}
In this Section we prove Theorem \ref{thm:gammaconvergences}. This establishes that \eqref{eq:introvarint} is the "correct'' approximation to the functional $||\nabla  u||_{\infty}$. The following definition is found in \cite{braides2002gamma}.
\begin{definition}
We say that the functional $E_n$  $\Gamma$-converges to $E$ if

1) (The $\Gamma$-$\liminf$)

Whenever $u_{n} \to u$ in $X$, we have
\begin{equation}
\label{eq:gamma_lim_inf}
\liminf_{n \to \infty}E_n(u_n) \geq E(u)
\end{equation}
2) (The $\Gamma$-$\limsup$)

For every $u\in X$, there exists a  sequence $u_n$ (called the recovery sequence) so that $u_n \to u$ and
\begin{equation}
\label{eq:gamma_lim_sup}
\limsup_{n \to \infty}E_n(u_n) \leq E(u)
\end{equation}
\end{definition}
Define
\begin{equation}
\label{eq:psigmavarint}
E_p^{\sigma}(u) = \left(\int_{\Omega}\{|\nabla  u|^2-\sigma\}_+^{p/2} \ dx\right)^{1/p} =\|\{|\nabla  u|^2-\sigma\}_+^{1/2}\|_p
\end{equation}
\begin{equation}
E_p(u) = \|\nabla  u\|_p,
\end{equation}
\begin{equation}
E_{\infty}^{\sigma}(u) = \|\{|\nabla  u|^2-\sigma\}^{1/2}_+\|_{\infty},
\end{equation}
and
\begin{equation}
\label{eq:esssup}
E_{\infty}(u)= \|\nabla  u\|_{\infty}
\end{equation}
\begin{theorem}
\label{thm:gammaconvptoinfty}
$\|\{|\nabla  u|^2-\sigma\}_+^{1/2}\|_p $ $\Gamma$-converges to $\|\{|\nabla  u|-\sigma\}_+\|_{\infty}$ with respect to uniform convergence in $C(\Omega)$.
\end{theorem}
\begin{proof}

Assume that $(u_p)_p\subset W^{1, p}(\Omega)$ is such that $\|\nabla  u_{p_j}\|_{p_j}\leq C$ for some subsequence $p_j\to \infty$ as $j \to \infty$. Our goal is to extract a subsequence of $u_p$ that converges uniformly to a function $u \in C(\Omega)$.

Fix $q>2$. Then Hölder's inequality gives the estimate
\begin{equation}
\label{eq:holderestimate}
\|\{|\nabla  u_{p_j}|^2-\sigma\}_+^{1/2}\|_q \leq |\Omega|^{\frac{1}{q}-\frac{1}{p_j}}\|\{|\nabla  u_{p_j}|^2-\sigma\}_+^{1/2}\|_{p_j}
\end{equation}
for all $p_j \geq q$.
Further, the Poincaré inequality gives
\[
\|u_{p_j}- u_{{\Omega}, p_j}\|_{W^{1, q}(\Omega)} \leq C\|\nabla  u_{p_j}\|_{W^{1, q}},
\]
so that the sequence $\|u_{p_j}- u_{{\Omega}, p_j}\|_{W^{1, q}(\Omega)} $ is bounded. Hence the weak compactness of the ${W^{1, q}}$-spaces implies the existence a subsequence that converges weakly in ${W^{1, q}(\Omega)}$ to some $u_q$. A diagonal procedure now gives a new subsequence, labeled $u_k$ for convenience, so that
\[
(u_k-u_{k,\Omega}) \to u \ \text{as} \ k \to \infty
\]
weakly in ${W^{1, q}(\Omega)}$ for all $2<q<\infty$. This implies that the limit function $u$ is in ${W^{1, q}(\Omega)}$ for all $q$.
The lower semi-continuity of the $q$-norm gives
\[
\|\{|\nabla  u|^2-\sigma\}_+^{1/2}\|_q \leq \liminf_{k \to \infty}\|\{|\nabla  u_k|^2-\sigma\}_+^{1/2}\|_q, 
\]
which together with the estimate \eqref{eq:holderestimate} gives
\begin{align*}
\|\{|\nabla  u|^2-\sigma\}_+^{1/2}\|_q &\leq |\Omega|^{\frac{1}{q}}\liminf_{k \to \infty}\left(|\Omega|^{-\frac{1}{k}}\|\{|\nabla  u_k|^2-\sigma\}_+^{1/2}\|_k\right)  \\
&= |\Omega|^{\frac{1}{q}}\liminf_{k \to \infty}\|\{|\nabla  u_k|^2-\sigma\}_+^{1/2}\|_k .
\end{align*}
As $q \to \infty$, we get
\begin{equation}
\label{eq:liminfestimate}
\|\{|\nabla  u|^2-\sigma\}_+^{1/2}\|_{\infty}  \leq \liminf_{k \to \infty}\|\{|\nabla  u_k|^2-\sigma\}_+^{1/2}\|_k.
\end{equation}
We see that we have $\{|\nabla  u|^2-\sigma\}_+^{1/2} \in L^{\infty}(\Omega)$, and so $u \in W^{1, \infty}(\Omega)$ and $u \in C(\Omega)$.

Fix a $q>n$, and let $V \subset \Omega$ be a sub-domain with regular boundary. Morrey's inequality and then Poincaré inequality gives
\begin{align*}
\|u_k-u_{k, \Omega}\|_{C^{0, 1-\frac{n}{q}}(V)} 
\leq C(q, n)\|u_k-u_{k, \Omega}\|_{W^{1, q}(V)} \\
\leq C(q, n)\|u_k-u_{k, \Omega}\|_{W^{1, q}(\Omega)} \\
\leq C(q, n)\|\nabla  u_k\|_q \leq K <\infty,
\end{align*}
for all $k \geq q$. Thus there exists a subsequence of $u_k$ that converges in $L^{\infty}(V)$ to $u$. Exhausting $\Omega$ with an increasing sequence of regular sets, a diagonal argument gives
\begin{equation}
\label{eq:uniformconvergence}
(u_k-u_{k, \Omega}) \to u \ \text{in} \ L^{\infty}(\Omega).
\end{equation}
We shall prove the $\Gamma$-$\liminf$ property, that is that for every sequence $u_p$ that converges uniformly to $u$ in $C(\Omega)$, we have that
\begin{equation}
\label{eq:pgammaliminf}
E_{\infty}^{\sigma}(u) \leq \liminf_{p \to \infty}E^{\sigma}_p(u_p).
\end{equation}
This follows directly from the estimate \eqref{eq:liminfestimate}, together with the uniform convergence of $u_p$ in \eqref{eq:uniformconvergence}.

Further we prove the $\Gamma$-$\limsup$ property, that is for every $u \in C(\Omega)$ there exists a sequence $u_p$ (called the recovery sequence of $u$) converging uniformly to $u$ so that
\begin{equation}
\label{eq:pgammalimsup}
E_{\infty}^{\sigma}(u) \geq \limsup_{p \to \infty}E^{\sigma}_p(u_p).
\end{equation}
Since 
\[
\|f\|_{L^{\infty}(\Omega)} = \lim_{p \to \infty}\|f\|_{L^p(\Omega)}
\]
holds for all measurable functions $f$, the $\Gamma$-$\limsup$ property follows immediately with $f =\{|\nabla  u|-\sigma\}_+^{\frac{1}{2}}$ and $u_p =u$ for all $p$.
\end{proof}
The fundamental theorem properties of $\Gamma$-convergence gives, \\ see \cite{braides2002gamma}:
\begin{enumerate}
\item If $\lim_{p \to \infty}(E^{\sigma}_p(u_p)-\min E^{\sigma}_p)=0$ then $u_p \to u$ in $C(\Omega)$, and $E_{\infty}^{\sigma}(u)=\min E_{\infty}^{\sigma}$.
\item If $E_{\infty}^{\sigma}(u)=\min E_{\infty}^{\sigma}$, then there exists a sequence $u_p$ with $u_p \to u$ as $p \to \infty$ so that $\lim_{p \to \infty}(E^{\sigma}_p(u_p)-\min E^{\sigma}_p)=0$
\end{enumerate}
This implies that any sequence $u_p$ of viscosity solutions of \eqref{eq:strongeq}  accumulate at a minimiser of $E_{\infty}^{\sigma}$. Using this, we can prove the following analogue to the classical Absolutely Minimizing Lipschitz Extension property of $\infty$-harmonic functions described in  \cite{bhattacharya1989limits}.
\begin{theorem}
\label{thm:absolutelyminsigma}
Let $u$ be the limit of minimizers. Then for every \\ $V \subset \Omega\setminus \{|\nabla  u|^2<\sigma\}$  we have that 
\begin{equation}
\|\{|\nabla  u|-\sigma\}_+\|_{L^{\infty}(V)} \leq \|\{|\nabla  w|-\sigma\}_+\|_{L^{\infty}(V)}
\end{equation}
for every $w \in W^{1, \infty}(\Omega) \cap C(\overline{\Omega})$, $u=w$ on $\partial V$.
\end{theorem}
\begin{proof}
The proof mimics \cite{bhattacharya1989limits}.

Let $w \in W^{1, \infty}(V) \cap C(\overline{V})$ be given, and consider \\ $\{w>u\}\subset V$. Fix $\epsilon >0$ so that $\{w>u+\epsilon\}$ is an open, non-empty subset of $\{w>u\}$. In view of uniform convergence of $u_p$, fix $p$ big enough so that 
\begin{align*}
\{w>u+\epsilon\} \subset \{w>u_p+\epsilon/2\} \subset \{w>u\} 
\end{align*}
We get
\begin{align*}
\int_{\{w>u+\epsilon\}}\{|\nabla  u|^2-\sigma\}_+^{p/2} \ dx 
\leq \int_{\{w>u_p+\epsilon/2\}}\{|\nabla  u|^2-\sigma\}_+^{p/2} \ dx \\
\leq \int_{\{w>u_p+\epsilon/2\}}\{|\nabla(w-\epsilon/2)|^2-\sigma\}_+^{p/2} \ dx \\
\leq \|\{|\nabla  u|^2-\sigma\}_+\|^p_{L^{\infty}(\{w>u_p+\epsilon/2\})}|(\{w>u_p+\epsilon/2\}| \\
\leq \|\{|\nabla  u|^2-\sigma\}_+\|^p_{L^{\infty}(\{w>u+\epsilon\})}|(\{w>u+\epsilon\}|.
\end{align*}
Raising both sides of the inequality to $1/p$ and $\liminf_{p \to \infty}$, we get
\begin{align*}
 \|\{|\nabla  u|^2-\sigma\}_+\|_{L^{\infty}(\{w>u+\epsilon\})} \\
 \leq \liminf_{p \to \infty}\|\{|\nabla  w|^2-\sigma\}_+\|_{L^{\infty}(\{w>u\})}|(\{w>u\}|^{\frac{1}{p}} \\
 \leq \|\{|\nabla  w|^2-\sigma\}_+\|_{L^{\infty}(\{w>u+\epsilon\})}.
\end{align*}
Since $\epsilon$ was arbitrary, we get
\[
\|\{|\nabla  u|^2-\sigma\}_+\|_{L^{\infty}(\{w>u\})} \leq \|\{|\nabla  w|^2-\sigma\}_+\|_{L^{\infty}(\{w>u\})}.
\]
We have that 
\[
V=\{w>u\} \cup \{w<u\} \cup \{w=u\},
\]
and the argument above can be repeated with the set $\{w<u\}$. Since $\nabla  u =\nabla  w$ in $\{u=w\}$, we have
\[
\|\{|\nabla  u|-\sigma\}_+\|_{L^{\infty}(V)} \leq \|\{|\nabla  w|-\sigma\}_+\|_{L^{\infty}(V)}.
\]
\end{proof}
The proof that $E_p$ $\Gamma$-converges to $E_{\infty}$ is very similar. All the arguments in the proof of Theorem \ref{thm:gammaconvptoinfty} is true for $\sigma =0$, and so we get
\begin{theorem}
\label{thm:p to infty}
As $p \to \infty$, 
\[
E_p \overset{\Gamma}{\to} E_{\infty},
\]
with respect to uniform convergence.
\end{theorem}
We prove that as $\sigma \to 0$, we retrieve the well-known $p$-energy functionals related to the $p$-Laplace equation.
\begin{theorem}
Let $p>n$. Then
\[
E^{\sigma}_p \overset{\Gamma}{\to} E_p \ \text{as} \ \sigma \to 0,
\]
in the uniform convergence topology on $C(\overline{\Omega})$.
\end{theorem}
\begin{proof}
We have from before that $\|\nabla  u_{\sigma}\|_p$ and $\|u_{\sigma}\|_p$ are bounded. Since $p>n$ , Morrey's inequality implies that $u \in C(V)$ for a regular $V \subset \Omega$. Well-known bounds give that the sequence $u_{\sigma}$ is equicontinuous on $V$, and Arzelá-Ascoli compactness criterion implies that $u_{\sigma} \to u$ as $\sigma \to 0$ uniformly on $V$. Exhausting $\Omega$ with regular sets, a diagonal procedure gives $u_{\sigma} \to u$ as $\sigma \to 0$ uniformly on $\Omega$. 

To prove the $\Gamma$-$\liminf$ property \eqref{eq:gamma_lim_inf} we must show that for every $u_{\sigma}$ that converges uniformly to $u$ we have
\begin{equation}
\label{eq:fatou}
\|\nabla  u\|_p \leq \liminf_{\sigma \to 0} \|\{|\nabla  u_{\sigma}|^2-\sigma\}^{1/2}_+\|_p.
\end{equation}
Clearly $\{|\nabla  u_{\sigma}|^2-\sigma\}^{p/2}_+ \to |\nabla  u|^p$ as $\sigma \to 0$, and so Fatou's lemma gives \eqref{eq:fatou}.

For the $\Gamma$-$\limsup$, we define our recovery sequence by $u_{\sigma} =u$ for all $\sigma >0$. Clearly $u_{\sigma} \to u$, and 
\begin{equation}
\label{eq:obvious}
\{|\nabla  u|^2-\sigma\}^{1/2}_+ \leq |\nabla  u|,
\end{equation}
so raising both sides to the power $p/2$, integrating over $\Omega$ and taking $\limsup$, we get
\[
\limsup_{\sigma \to 0}\int_{\Omega}\{|\nabla  u|^2-\sigma\}_+^{p/2} \ dx \leq \int_{\Omega}|\nabla  u|^p \ dx,
\]
showing that property \eqref{eq:gamma_lim_sup} holds, and so $E^{\sigma}_p$ $\Gamma$-converges to $E_p$ with respect to uniform convergence.
\end{proof}
For the last convergence in Theorem \ref{thm:gammaconvergences}, we note that since \eqref{eq:fatou} holds for all $p>n$, it also holds in the limit $p\to \infty$. This combined with \eqref{eq:obvious}, shows that the following Theorem holds.
\begin{theorem}
\label{thm:lastthm}
As $\sigma \to 0$, 
\[
E_{\infty}^{\sigma} \overset{\Gamma}{\to} E_{\infty},
\]
with respect to uniform convergence.
\end{theorem} 
We have that Theorem \ref{thm:gammaconvptoinfty}, Theorem \ref{thm:p
  to infty}, Theorem \ref{thm:lastthm} and Theorem
\ref{eq:pgammaliminf} together prove Theorem
\ref{thm:gammaconvergences}. 

\section*{Acknowledgement}
The authors would also like to thank Professor Peter Lindqvist 
for his kind advice.  
\bibliographystyle{alpha}
\bibliography{refs} 
\end{document}